\documentclass[11pt]{amsart}
\usepackage{amsfonts}
\usepackage[colorlinks, bookmarks=true]{hyperref}
\usepackage{color,graphicx,shortvrb}
\usepackage{enumerate}
\usepackage{amsmath}
\usepackage{amssymb}

\newtheorem{thm}{Theorem}
\newtheorem{corollary}[thm]{Corollary}
\newtheorem{lemma}[thm]{Lemma}
\newtheorem{prop}[thm]{Proposition}
\theoremstyle{definition}

\theoremstyle{remark}
\newtheorem{remark}[thm]{Remark}
\newtheorem{problem}[thm]{Problem}

\numberwithin{equation}{section}

\def\J#1#2#3{ \left\{ #1,#2,#3 \right\} }

\usepackage[latin 1]{inputenc}

\begin{document}
\title{\v{C}eby\v{s}\"{e}v subspaces of JBW$^{*}$-triples}
\author[F.B. Jamjoom]{Fatmah B. Jamjoom}
\address{Department of Mathematics, College of Science, King Saud
University, P.O.Box 2455-5, Riyadh-11451, Kingdom of Saudi Arabia.}
\email{fjamjoom@ksu.edu.sa}
\author[A.M. Peralta]{Antonio M. Peralta}
\address{Departamento de An{\'a}lisis Matem{\'a}tico, Universidad de Granada,%
\\
Facultad de Ciencias 18071, Granada, Spain}
\curraddr{Visiting Professor at Department of Mathematics, College of
Science, King Saud University, P.O.Box 2455-5, Riyadh-11451, Kingdom of
Saudi Arabia.}
\email{aperalta@ugr.es}
\author[A.A. Siddiqui]{Akhlaq A. Siddiqui}
\address{Department of Mathematics, College of Science, King Saud
University, P.O.Box 2455-5, Riyadh-11451, Kingdom of Saudi Arabia.}
\email{asiddiqui@ksu.edu.sa}
\author[H.M. Tahlawi]{Haifa M. Tahlawi}
\address{Department of Mathematics, College of Science, King Saud
University, P.O.Box 2455-5, Riyadh-11451, Kingdom of Saudi Arabia.}
\email{htahlawi@ksu.edu.sa}

\begin{abstract} We describe the one-dimensional \v{C}eby\v{s}\"{e}v subspaces of a JBW$^*$-triple $M,$ by showing that for a non-zero element $x$ in $M$, $\mathbb{C}x$ is a \v{C}eby\v{s}\"{e}v subspace of $M$ if, and only if,  $x$ is a Brown-Pedersen quasi-invertible element in ${M}$. We study the \v{C}eby\v{s}\"{e}v JBW$^*$-subtriples of a JBW$^*$-triple $M$. We prove that, for each non-zero \v{C}eby\v{s}\"{e}v JBW$^*$-subtriple $N$ of $M$, then exactly one of the following statements holds:\begin{enumerate}[$(a)$]\item $N$ is a rank one JBW$^*$-triple with dim$(N)\geq 2$ (i.e. a complex Hilbert space regarded as a type 1 Cartan factor). Moreover, $N$ may be a closed subspace of arbitrary dimension and $M$ may have arbitrary rank;
\item $N= \mathbb{C} e$, where $e$ is a complete tripotent in $M$;
\item $N$ and $M$ have rank two, but $N$ may have arbitrary dimension;
\item $N$ has rank greater or equal than three and $N=M$.
\end{enumerate}\smallskip

We also provide new examples of \v{C}eby\v{s}\"{e}v subspaces of classic Banach spaces in connection with ternary rings of operators.
\end{abstract}

\thanks{The authors extend their appreciation to the Deanship of Scientific Research at King Saud University for funding this work through research group no  RG-1435-020. The second author also is partially supported by the Spanish Ministry of Economy and Competitiveness project no. MTM2014-58984-P.  }

\maketitle
\section{Introduction}

Let $V$ be a subspace of a Banach space $X$. The subspace $V$ is called a \emph{\v{C}eby\v{s}\"{e}v {\rm(}Chebyshev{\rm)} subspace} of $X$ if and only if for each $x\in X$ there exists a unique point $x_{\circ }\in V$ such that $\hbox{dist}(x,V)=\left\Vert x-x_{\circ }\right\Vert $.\smallskip

Let $K$ be a compact Hausdorff space. A classical theorem due to A. Haar establishes that an $n$-dimensional subspace $V$ of the space $C(K)$, of all continuous complex-valued functions on $K$, is a \v{C}eby\v{s}\"{e}v subspace of $C(K)$ if, and only if, any non-zero $f\in V$ admits at most $n-1$ zeros (cf. \cite{Haar} and the monograph \cite[p. 215]{Singer}). Having in mind the Riesz representation theorem, and the characterization of the extreme points of the closed unit ball in the dual space of $C(K)$, we can easily see that, in the above conditions, $V$ is an $n$-dimensional \v{C}eby\v{s}\"{e}v subspace of $C(K)$ if, and only if, for every set $\{\delta_{t_{1}},\ldots, \delta_{t_n}\}$ of $n$-mutually orthogonal pure states we have $\displaystyle V\cap \bigcap_{i=1}^{n} \ker(\delta_{t_i}) = \{0\}$. This result implies that any non-zero $f$ in $C(K)$ spans a \v{C}eby\v{s}\"{e}v subspace of the latter space if, and only if, $f$ is invertible in the algebra $C(K)$.\smallskip

Later on, J.G. Stampfli proved in \cite[Theorem 2]{Stampfli}, that the scalar multiples of the unit element in a von Neumann algebra $M$ is a \v{C}eby\v{s}\"{e}v subspace of $M$.  In \cite{LeggScrantonWard1975}, D.A. Legg, B.E. Scranton, and J.D. Ward characterize the semi-\v{C}eby\v{s}\"{e}v and finite dimensional \v{C}eby\v{s}\"{e}v subspaces of $K(H)$, the algebra of compact operators on an infinite-dimensional Hilbert space $H$. They conclude that, for a separable Hilbert space $H$, there exist \v{C}eby\v{s}\"{e}v subspaces of every finite dimension in $K(H)$ \cite[Theorem 3]{LeggScrantonWard1975}, when $H$ is not separable $K(H)$ has no finite-dimensional \v{C}eby\v{s}\"{e}v subspaces \cite[Corollary 2]{LeggScrantonWard1975}.\smallskip

A.G. Robertson continued with the study on \v{C}eby\v{s}\"{e}v subspaces of von Neumann algebras in  \cite{Robertson1}, where he established the following results:

\begin{thm}\label{t Robertson 1dimensional Cebysev subspaces in von Neumann algebras}{\rm(}\cite[Theorem 6]{Robertson1}{\rm)} Let $x$ be a non-zero element in a von Neumann algebra $M$.
Then, the one dimensional subspace $\mathbb{C} x$ is a \v{C}eby\v{s}\"{e}v subspace of ${M}$ if and only if there is a projection $p$ in the
center of ${M}$ such that $p x$ is left invertible in $p{M}$ and $(1-p)x$ is right invertible in $(1-p){M}$.
\end{thm}

\begin{thm}\label{t Robertson Cebysev finite dimensional von Neumann subalgebras in von Neumann algebras}{\rm(}\cite[Theorem 6]{Robertson1}{\rm)} Let $N$ be finite dimensional $^*$-subalgebra of an infinite dimensional von Neumann algebra ${M}$.
Suppose $N$ has dimension $>1$. Then $N$ is not a \v{C}eby\v{s}\"{e}v subspace of ${M}$.
\end{thm}

A.G. Robertson and D. Yost prove in \cite[Corollary 1.4]{Robertson2} that an infinite dimensional C$^*$-algebra $A$ admits a finite dimensional $^*$-subalgebra $B$ which is also a \v{C}eby\v{s}\"{e}v in $A$ if and only if $A$ is unital and $B=\mathbb{C} 1$.\smallskip

The results proved by Robertson and Yost were complemented by G.K. Pedersen, who shows that if $A$ is a C$^*$-algebra without unit and $B$ is a \v{C}eby\v{s}\"{e}v C$^*$-subalgebra of $A$, then $A=B$ (compare \cite[Theorem 4] {Pederesen}).\smallskip

The previous results of Robertson \cite{Robertson1} and Pedersen \cite[Theorem 2]{Pederesen} also prove the following equivalent reformulation of Theorem \ref{t Robertson 1dimensional Cebysev subspaces in von Neumann algebras}: for each non-zero element $x$ in a von Neumann algebra $M$, the following statements are equivalent:\label{equ reformulation Ronbertson Pedersen}\begin{enumerate}[$(a)$] \item $\mathbb{C} x$ is a \v{C}eby\v{s}\"{e}v subspace of ${M}$;
\item $x$ is Brown-Pedersen quasi-invertible in ${M}$;
\item For each pure state (i.e. for each extreme point of the positive part of the closed unit ball of $M^*$) $\varphi\in M^*$, and for each unitary $u\in M$, we have $\varphi (x^* x) +\varphi (u x x^* u ) >0$.
\end{enumerate}

Then, the one dimensional subspace $\mathbb{C} x$ is a \v{C}eby\v{s}\"{e}v subspace of ${M}$ if and only if there is a projection $p$ in the
center of ${M}$ such that $p x$ is left invertible in $p{M}$ and $(1-p)x$ is right invertible in $(1-p){M}$.\smallskip

A renewed interest on \v{C}eby\v{s}\"{e}v subspaces of C$^*$-algebras has led M. Namboodiri, S. Pramod, and A. Vijayarajan to revisit and generalize the previous contributions of Robertson, Yost and Pedersen in \cite{Namboodiri}.\smallskip

On the other hand, C$^*$-algebras can be regarded as elements in a strictly wider class of complex Banach spaces called JB$^*$-triples (see \S 2 for the detailed definitions). Many geometric properties studied in the setting of C$^*$-algebras have been also explored in the bigger class of JB$^*$-triples. However \v{C}eby\v{s}\"{e}v subspaces and the theory of best approximations remains unexplored in the class of JB$^*$-triples. In this note we present the first results about \v{C}eby\v{s}\"{e}v subspaces and \v{C}eby\v{s}\"{e}v subtriples in Jordan structures.\smallskip

In Section \ref{sec: one dimensional Cebysev subspaces} we prove that for a non-zero element $x$ in a JBW$^*$-triple $M$, $\mathbb{C}x$ is a \v{C}eby\v{s}\"{e}v subspace of $M$ if, and only if,  $x$ is a Brown-Pedersen quasi-invertible element in ${M}$ (see Theorem \ref{t one dimensional Cebysev subspaces in a JBW*-triple}). This result generalizes the result established by Robertson in Theorem \ref{t Robertson 1dimensional Cebysev subspaces in von Neumann algebras} (cf. \cite{Robertson1}), but it also add a new perspective from an independent argument.\smallskip

In Section \ref{sec: Cebysev JBW*-subtriples} we establish a precise description of the JBW$^*$-subtriples of a JBW$^*$-triple $M$ which are
\v{C}eby\v{s}\"{e}v  subspaces in $M$. We should remark that in the setting of von Neumann algebras and C$^*$-algebras, the scarcity of non-trivial \v{C}eby\v{s}\"{e}v $^*$-subalgebras is endorsed with the following results: If an infinite dimensional von Neumann algebra, $M$, contains a finite dimensional von Neumann subalgebra $N$ which is a \v{C}eby\v{s}\"{e}v subspace in $M$, then $N$ must be one dimensional (compare Theorem \ref{t Robertson Cebysev finite dimensional von Neumann subalgebras in von Neumann algebras} or \cite[Theorem 6]{Robertson1}). Furthermore, an infinite dimensional C$^*$-algebra $A$ admits a finite dimensional $^*$-subalgebra $B$ which is also a \v{C}eby\v{s}\"{e}v in $A$ if and only if $A$ is unital and $B=\mathbb{C} 1$ (cf. \cite[Corollary 1.4]{Robertson2}). If $A$ is a C$^*$-algebra without unit and $B$ is a \v{C}eby\v{s}\"{e}v C$^*$-subalgebra of $A$, then $A=B$ (compare \cite[Theorem 4] {Pederesen}). The first main difference in the setting of JB$^*$-triples is the existence of \v{C}eby\v{s}\"{e}v JB$^*$-subtriples with arbitrary dimensions; complex Hilbert spaces and spin factors give a complete list of examples (compare Remark \ref{remark spin} and comments before it).\smallskip

In our main result about \v{C}eby\v{s}\"{e}v JBW$^*$-subtriples (cf. Theorem \ref{t Cebysev  JBW*-subtriples}), we establish the following criterium: Let $N$ be a non-zero \v{C}eby\v{s}\"{e}v JBW$^*$-subtriple of a JBW$^*$-triple $M$. Then exactly one of the following statements holds:\begin{enumerate}[$(a)$]\item $N$ is a rank one JBW$^*$-triple with dim$(N)\geq 2$ (i.e. a complex Hilbert space regarded as a type 1 Cartan factor). Moreover, $N$ may be a closed subspace of arbitrary dimension and $M$ may have arbitrary rank;
\item $N= \mathbb{C} e$, where $e$ is a complete tripotent in $M$;
\item $N$ and $M$ have rank two, but $N$ may have arbitrary dimension;
\item $N$ has rank greater or equal than three and $N=M$.
\end{enumerate}

We provide examples of infinite dimensional proper \v{C}eby\v{s}\"{e}v JBW$^*$-subtriples of JBW$^*$-triples (see Remark \ref{remark spin}). We apply the solution of the minimum covering sphere problem in the Euclidean space $\ell_2^m$ to present new examples of \v{C}eby\v{s}\"{e}v subspaces of classical Banach spaces (cf. Remark \ref{example Hilbert Cebysev in rectangular Banach spaces}), and to construct an example of a rank-one Hilbert space which is a \v{C}eby\v{s}\"{e}v JBW$^*$-subtriple of a rank-$n$ JBW$^*$-triple, where $n$ is an arbitrary natural number (cf. Remark \ref{example Hilbert Cebysev in rectangular}).\smallskip

It should be remarked at this point that the techniques applied by Robertson, Yost \cite{Robertson1,Robertson2} and Pedersen \cite{Pederesen} in the setting of von Neumann algebras do not make any sense in the wider setting of JBW$^*$-triples. The techniques developed in this paper are completely independent and provide new arguments to understand the \v{C}eby\v{s}\"{e}v von Neumann subalgebras of a von Neumann algebra (Corollary \ref{cor JBW*-subtriples of von Neumann algebras}).

\section{One-dimensional \v{C}eby\v{s}\"{e}v subspaces and subtriples of JBW$^*$-triples}\label{sec: one dimensional Cebysev subspaces}

A complex Jordan triple system is a complex linear space $E$ equipped with a triple product which is bilinear and symmetric in the external variables and conjugate linear in the middle one and satisfies the Jordan identity:
\begin{equation}\label{eq Jordan equation} L(x,y)\{a,b,c\}=\{L(x,y)a,b,c\}-\{a,L(y,x)b,c\}+\{a,b,L(x,y)c\},
\end{equation}

for all $x,y,a,b,c\in E$, where $L(x,y):E\rightarrow E$ is the linear
mapping given by $L(x,y)z=\{x,y,z\}$.\smallskip

A \emph{JB$^{*}$-triple} is a complex Jordan triple system $E$ which is a Banach
space satisfying the additional ``\emph{geometric}'' axioms:\begin{enumerate}[$(a)$]\item  For each $x\in E$, the operator $L(x,x)$ is hermitian with non-negative spectrum;
\item $\left\Vert \{x,x,x\}\right\Vert =\left\Vert x\right\Vert ^{3}$ for all $%
x\in E$.
\end{enumerate}\smallskip

Every C$^*$-algebra is a JB$^*$-triple with respect to the triple product given by \begin{equation}\label{eq ternary product on C*-algebras} \{a,b,c\} = \frac12 (a b^* c+ c b^* a).
\end{equation} Every JB$^*$-algebra (i.e. a complex Jordan Banach $^*$-algebra satisfying $$\|U_a ({a^*}) \|= \|a\|^3,$$ for every element $a$, where $U_a (x) :=2 (a\circ x) \circ a - a^2 \circ x$, cf. \cite[\S 3.8]{HancheStor}) is a JB$^*$-triple under the triple product defined \begin{equation}\label{eq JB*-alg-triple product} \J xyz = (x\circ y^*) \circ
z + (z\circ y^*)\circ x - (x\circ z)\circ y^*.
\end{equation} The space $B(H,K)$ of all bounded linear operators between complex Hilbert spaces, although rarely is a C$^*$-algebra, is a JB$^*$-triple with the product defined in \eqref{eq ternary product on C*-algebras}. In particular, every complex Hilbert space is a JB$^*$-triple.\smallskip

Other examples of JB$^*$-triples are given by the so-called \emph{Cartan factors}. A Cartan factor of type 1 is a JB$^*$-triple which coincides with the Banach space $B(H, K)$ of bounded linear operators between two complex Hilbert spaces, $H$ and $K$, where the triple product is defined by \eqref{eq ternary product on C*-algebras}. Cartan factors of types 2 and 3 are JB$^*$-triples which can be identified the subtriples of $B(H)$ defined by $II^{\mathbb{C}} = \{ x\in B(H) : x=- j x^* j\} $ and $III^{\mathbb{C}} = \{ x\in B(H) : x= j x^* j\}$, respectively, where $j$ is a conjugation on $H$. A Cartan factor of type 4 or $IV$ is a spin factor, that is, a complex Hilbert space provided with
a conjugation $x \mapsto \overline{x}$, where the triple product and the norm are defined by $$\J x y z
= \langle x / y \rangle z + \langle z / y \rangle x - \langle x /
\bar z \rangle \bar y,$$ and $\| x\|^2=\langle x / x
\rangle+\sqrt {\langle x / x \rangle^2-|\langle x / \overline x
\rangle|^2}$, respectively. The Cartan factors of types 5 and 6 consist of finite dimensional spaces of matrices over the eight dimensional complex Cayley division algebra $\mathbb{O}$; the type $VI$ is the space of all hermitian $3$x$3$ matrices over $\mathbb{O}$, while the type $V$ is the subtriple of $1$x$2$ matrices with entries in $\mathbb{O}$ (compare \cite{Loos77}, \cite{FriRu86}, and \cite[\S 2.5]{Chu}).\smallskip

A JB$^{*}$-triple $W$ is called a \emph{JBW$^{*}$-triple} if it has a predual $W_{\ast }$. It is known that a JBW$^*$-triple admits a unique isometric predual and its triple product is separately $\sigma (W,W_{\ast })$-continuous (see \cite{Barton}). The second dual $E^{\ast \ast }$ of a JB$^{*}$-triple $E$ is a
JBW$^{*}$-triple with respect to a triple product which extends the triple product of $E$ (cf. \cite{Di86b}).\smallskip

For more detail of the properties of JB$^{*}$-triples and JBW$^{*}$-triples the reader is referred to the monographs \cite{Chu} and \cite{CabRod2014}.\smallskip

Given an element $a$ in a JB$^{*}$-triple $E$, the symbol $Q(a)$ will denote the conjugate linear operator on $E$ defined by $Q(a)(x)=\{a,x,a\}$.\smallskip

An element $e\in E$ is called \textit{a tripotent} when $\{e,e,e\}=e$. Each tripotent $e\in E$ induces a decomposition of $E$, called \textit{the Peirce decomposition,} in the form $E=E_{2}(e)\oplus
E_{1}(e)\oplus E_{0}(e)$, where $E_{i}(e)$ is the $\frac{i}{2}$ eigenspace of the operator $L(e,e)$, $i=0,1,2$. This decomposition satisfies the following \emph{Peirce rules:} $$\{E_{2}(e),E_{0}(e),E\}=\{%
E_{0}(e),E_{2}(e),E\}=0$$ and $$\{E_{i}(e),E_{j}(e),E_{k}(e)\}\subseteq E_{i-j+k}(e),$$ when $i-j+k\in \{0,1,2\}$ and is zero otherwise. The projection $P_{k}(e)$ of $E$ onto $E_{k}(e)$ is called the \textit{Peirce} $k$\textit{-projection}. It is known that Peirce projections are contractive (cf. \cite[Corollary 1.2]{FridRusso}) and satisfy: $$P_{2}(e)=Q(e)^{2},\ P_{1}(e)=2(L(e,e)-Q(e)^{2}),$$ and  $$ P_{0}(e)=Id_{E}-2L(e,e)+Q(e)^{2} .$$

The separate weak$^*$-continuity of the triple product of a JBW$^*$-triple $M$ implies that Peirce projections associated with a tripotent $e$ in $M$ are weak$^*$-continuous.\smallskip

It is known that the Peirce-2 subspace $E_{2}(e)$ is a JB$^{*}$-algebra with unit $e$, Jordan product $x\circ _{e}y:=\{x,e,y\}$ and involution $x^{\ast _{e}}:=\{e,x,e\}$, respectively. Since surjective linear isometries and triple isomorphisms on a JB$^*$-triple coincide (cf. \cite[Proposition 5.5]{Ka83}), the triple product in $E_{2}(e)$ is uniquely given by $$\{x,y,z\}=(x\circ _{e}y^{\ast _{e}})\circ _{e}z+(z\circ _{e}y^{\ast
_{e}})\circ _{e}x-(x\circ _{e}z)\circ _{e}y^{\ast _{e}},$$ $x,y,z\in E_{2}(e)$.
\smallskip

We shall make use of the following property: given a tripotent $e\in E$ and an element $\lambda$ in the unit sphere of $\mathbb{C}$, the mapping: \begin{equation}\label{eq mapping Slambda (e)} S_{\lambda} (e) : E\to E, \ S_{\lambda} (e) = \lambda^2 P_2 (e)+ \lambda P_1 (e) + P_0 (e),
 \end{equation} is a surjective linear isometry on $E$ and a triple isomorphism (compare \cite[Lemma 1.1]{FridRusso}).\smallskip

A tripotent $e\in E$ is said to be \textit{unitary} if the operator $L(e,e)$ coincides with the identity map $I_{E}$ on $E$; that is, $E_{2}(e)=E$. We shall say that $e$ is \textit{complete} or \emph{maximal} when $E_{0}(e)=E$.  When $E_2 (e) = P_{2}(e) (E) =\mathbb{C}e\neq \{0\}$, we say that $e$ is \textit{minimal}. % if , and it is called \textit{abelian} if $W_{e}$ is an abelian Jordan triple system, that is $[L(a,b),L(x,y)]=0$ for all $a,b,x,y\in W_{e}$ \cite[p.211]{Chu}.
\smallskip

The complete tripotents of a JB$^{\ast }$-triple $E$ coincide with the real and complex extreme points of its closed unit ball $E_{1}$ (cf. \cite[Lemma 4.1]{BraKaUp78} and \cite[Proposition 3.5]{KaUp77} or \cite[Theorem 3.2.3]{Chu}). Consequently, the Krein-Milman theorem assures that every JBW$^*$-triple admits an abundant set of complete tripotents \cite[Corollary 3.2.4]{Chu}\label{page complete trip}.\smallskip

When $a$ is an element in a JBW$^{*}$-triple $M$, the sequence $(a^{\frac{1}{2n-1}})$ converges in the weak$^{\ast }$-topology
of $M$ to a tripotent, denoted by $r(a)$, called the \textit{range tripotent of} $a$. The tripotent $r(a)$ is the smallest tripotent $e\in M$ satisfying
that $a$ is positive in the JBW$^{\ast }$-algebra $M_{2}(e)$ (see \cite[page 322]{Edwards}).\smallskip

Let $a$ be an element in a JB$^*$-triple $E$. It is known that the JB$^*$-subtriple $E_{a}$ generated by $a$, identifies with some $C_0(L)$ where $\|a\|\in L\subseteq [0,\|a\|]$ with $L\cup \{0\}$ compact (cf. \cite[1.15]{Ka83}). Moreover, there exists a triple isomorphism $\Psi: E_{a} \to C_0(L)$ such that $\Psi (a) (t) =t$. Clearly, the range tripotent $r(a)$ can be identified with the characteristic function $\chi_{_{(0,\|a\|]\cap L}}\in C_0(L)^{**}$ (see \cite[beginning of \S 2]{BuChuZa2000}).\label{page triple functional calculus}\smallskip

We recall that an element $x$ in a Jordan algebra $\mathcal{J}$ with unit $e$ is called \textit{invertible} if there exists an element $y$ such that $x\circ y=e$
and $x^{2}\circ y=x$. The element $y$ is called \textit{the inverse} of $x$, and is denoted by $x^{-1}$. Inverse of any element $x$ in a Jordan algebra $\mathcal{J}$ is unique whenever it exists. The set of all invertible elements in $\mathcal{J}$ is denoted by $\mathcal{J}^{-1}$. \smallskip

An element element $a$ in a JB$^{*}$-triple $E$ is called \textit{von Neumann regular} if and only if there exists $b\in E$ such that $$Q(a)(b)=a,\  Q(b)(a)=b,\hbox{ and } [Q(a),Q(b)]:=Q(a)Q(b)-Q(b)Q(a)=0.$$ When $a$ is von Neumann regular, the (unique) element $b\in E$ satisfying the above conditions is called \textit{the generalized} \textit{inverse of} $a$, and is denoted by $a^{\dag }$. It is known that an element $a\in E$ is von Neumann regular if, and only if, $Q(a)$ has norm-closed image if, and only if, the range tripotent $r(a)$ of $a$ lies in $E$ and $a$ is positive and invertible element of the JB$^{*}$-algebra $E_{2}(r(a))$ (compare \cite{Burgos3}). Furthermore, when $a$ is von Neumann regular, $Q(a)Q(a^{\dag })=Q(a^{\dag })Q(a)=P_{2}(r(a))$ and $L(a,a^{\dag })=L(a^{\dag },a)=L(r(a),r(a))$ \cite[page 192]{Burgos3}.\smallskip

Given a pair of elements $a,b$ in a JB$^{*}$-triple $E$, \textit{the Bergmann operator} associated to $a$ and $b$ is the mapping $B(a,b):E \rightarrow L(E)$ defined by $B(a,b)=Id_{E}-2L(a,b)+Q(a)Q(b)$ (cf. \cite[page 22]{Chu}).\smallskip

An element $a$ in a JB$^{*}$-triple $E$ is said to be \textit{Brown-Pedersen quasi-invertible} (\emph{BP-quasi-invertible} for short) when it is von Neumann regular with generalized inverse $b$ such that the Bergman operator $B(a,b)$ vanishes; in such a case, $b$ is called \textit{the BP-quasi inverse} of $a$. The set of BP-quasi invertible elements in $E$ is denoted by $E_{q}^{-1}$ \cite{Tahlawi}. It is established in \cite{Tahlawi} that an element $a\in E$ is BP-quasi-invertible if, and only if, one of the following equivalent
statements holds:\begin{enumerate}[$(i)$]\item $a$ is von Neumann regular, and its range tripotent $r(a)$ is an
extreme point of the closed unit ball $E_{1}$ of $E$
(i.e. $r(a)$ is a complete tripotent of $E$);
\item There exists a complete tripotent $e\in E$ such that $a$ is
positive and invertible in the JB$^{\ast }$-algebras $E_{2}(e)$.
\end{enumerate}

We recall that two elements $a,b$ in a JB$^*$-triple, $E,$ are
said to be \emph{orthogonal} (written $a\perp b$) if $L(a,b) =0$.
Lemma 1 in \cite{Burgos1} shows that $a\perp b$ if and only
if one of the following nine statements holds:

\begin{equation}
\label{ref orthogo}\begin{array}{ccc}
  \J aab =0; & a \perp r(b); & r(a) \perp r(b); \\
  & & \\
  E^{**}_2(r(a)) \perp E^{**}_2(r(b));\ \ \ & r(a) \in E^{**}_0 (r(b));\ \ \  & a \in E^{**}_0 (r(b)); \\
  & & \\
  b \in E^{**}_0 (r(a)); & E_a \perp E_b & \J bba=0.
\end{array}
\end{equation}

Let $e$ be a tripotent in a JB$^*$-triple $E$. Lemma 1.3$(a)$ in \cite{FridRusso} shows that $$\|x_2 + x_0\| = \max \{ \|x_2\|, \|x_0\| \},$$ for every $x_2\in E_2 (e) $ and every $x_0\in E_0 (e).$ Combining this result with the equivalences in \eqref{ref orthogo} we see that \begin{equation}\label{eq orthogonal are M-orthogonal} \|a + b\|  = \max \{ \|a\|, \|b\| \},
\end{equation} whenever $a$ and $b$ are orthogonal elements in a JB$^*$-triple.\smallskip

Given a subset $M\subseteq E$, we write $M_{E}^{\perp }$
(or simply $M^{\perp }$) for the (orthogonal) annihilator of $M$ defined by
$M_{E}^{\perp }=\{y\in E:y\perp x,\forall x\in M\}$. If $e\in E$ is a tripotent,
then $\{e\}^{\perp }=$ $E_{0 }(e)$, and $\{a\}^{\perp }=$ $(E^{\ast \ast })_{0 }(r(a))\cap
E$, for every $a\in E$ (cf. \cite[Lemma 3.2]{Burgos2}).\smallskip

\begin{lemma}\label{l Cebysev implies BP-q invert} Let $V$ be a non-zero \v{C}eby\v{s}\"{e}v subspace of a JBW$^{\ast }$-triple $M$. Then $V\cap M_{q}^{-1}\neq \emptyset$, where $M_{q}^{-1}$ denotes the set of BP-quasi invertible elements of $M$.
\end{lemma}

\begin{proof} Arguing by contradiction, we suppose that $V\cap M_{q}^{-1}= \emptyset$.\smallskip

Let us take $x\in V$ with $\Vert x\Vert =1$. By assumptions, $x\notin M_{q}^{-1}$. Under these conditions, the range complete tripotent of $x$, $r(x)$ is not complete in $M$ or $x$ is not invertible in the JBW$^*$-algebra $M_2 (r(x))$. By \cite[Lemma 3.12]{Horn}, there exists a complete tripotent $e$
in $M$ such that $r(x)\leq e$.\smallskip

We shall identify the JB$^*$-subtriple, $M_x$, of $M$ generated by $x$ with some $C_0(L)$ where $1=\|x\|\in L\subseteq [0,\|1\|]$ with $L\cup \{0\}$ compact (cf. \cite[1.15]{Ka83}). We further know that there exists a triple isomorphism $\Psi: M_{x} \to C_0(L)$ such that $\Psi (x) (t) =t$, and the range tripotent $r(x)$ identifies with the characteristic function $\chi_{_{(0,\|x\|]\cap L}}\in C_0(L)^{**}$ (see page \ref{page triple functional calculus}). It is clear that, under this identification, $$\|r(x)-\lambda x\| =  1-|\lambda| \inf\{|x(t)|: t\in L\}\leq 1,$$ for every $|\lambda|\leq 1$ in $\mathbb{C}.$ When $e= r(x),$ the element $x$ is not invertible in the JBW$^*$-algebra $M_2 (r(x))$, and hence $\|e-x\| = \|r(x)-x\| =1$. When $e\gneqq r(x)$, we have $\|e-r(x) \|=1$. Thus, applying $e-r(x) \perp r(x)$ and \eqref{eq orthogonal are M-orthogonal}, we further known that $$\|e- \lambda x\| = \|e-r(x) + r(x) -\lambda x\| = \max \{\|e-r(x)\|, \|r(x) -\lambda x\|\} =1.$$

We observe that, since $e$ is a complete tripotent, $e\in M_q^{-1}$, and hence $e\notin V$. Since $V$ is a \v{C}eby\v{s}\"{e}v subspace, there exists a unique best approximation, $c_{_V} (e)\in V,$ of $e$ in $V$ satisfying $\hbox{dist}(e,V)=\left\Vert e- c_{_V} (e)\right\Vert >0$.\smallskip

If $\hbox{dist}(e,V)=\left\Vert e- c_{_V} (e)\right\Vert \geq 1,$ we would have $1 = \|e\|\geq \hbox{dist}(e,V)= 1,$ and
$$1=\left\Vert e- c_{_V} (e)\right\Vert= \hbox{dist}(e,V) =  \Vert e-\lambda x\Vert,$$ for every $|\lambda|\leq 1$, contradicting the uniqueness of the best approximation of $e$ in $V$. We can therefore assume that $\hbox{dist}(e,V)<1.$ Consequently, there exits $y\in V$ with $\|e-y\|<1.$ Corollary 2.4. in \cite{Jamjoom} implies that $y\in M_{q}^{-1}\cap V$, which is impossible.
\end{proof}

Let $e$ be a tripotent in a JB$^*$-triple $E$. Let us recall that $e$ is a tripotent in the JBW$^*$-triple $E^{**},$ and that Peirce projections associated with $e$ on $E^{**}$ are weak$^*$-continuous. Goldstine's theorem assures that $E$ is weak$^*$-dense in $E^{**}$, and hence, $E^{**}_k (e)$ coincides with the weak$^*$-closure of $E_k (e)$ in $E^{**}$, for every $k=0,1,2.$ In particular, $e$ is complete in $E^{**}$ whenever $e$ is a complete tripotent in $E.$ Moreover, since the orthogonal complement of a tripotent $e$ in a JB$^*$-triple $F$ coincides with $F_0 (e)$, we have:

\begin{lemma}\label{l complete tripotents do not admit orthogonal complements}
Let $e$ be a complete tripotent in a JB$^*$-triple $E$. Then $\{e\}^{\perp}_{_{E^{**}}}=\{0\}$, that is, $e$ is not orthogonal to any non-zero element in $E^{**}$.$\hfill\Box$
\end{lemma}

The following technical result is part of the folklore in the theory of best approximation (see \cite[Lemma 3]{Robertson1} or \cite[Theorem 2.1]{Singer}).

\begin{lemma}\label{l tech Roeberton Singer}{\rm(}\cite[Lemma 3]{Robertson1}{\rm)}. Let $x$ be an element in complex a Banach space $X$ such that $\mathbb{C}x$ is not a \v{C}eby\v{s}\"{e}v subspace of $X$. Then there exists an extreme
point $\phi $ of the closed unit ball of $X^{*}$, a vector $y\in X$ and a scalar $\lambda \in \mathbb{C}\backslash \{0\}$ such that \begin{enumerate}[$(a)$] \item $\phi (x)=0$;
\item $\phi (y)=$ $\left\Vert y\right\Vert =\left\Vert y-\lambda
x\right\Vert $. $\hfill\Box$
\end{enumerate}
\end{lemma}

We can characterize now the one dimensional \v{C}eby\v{s}\"{e}v subspaces of a JBW$^*$-triple.

\begin{thm}\label{t one dimensional Cebysev subspaces in a JBW*-triple}
Let $x$ be a non-zero element in a JBW$^*$-triple $M$. The following statements are equivalent:
\begin{enumerate}[$(a)$] \item $\mathbb{C}x$ is a \v{C}eby\v{s}\"{e}v subspace of
$M$;
\item $x$ is a Brown-Pedersen quasi-invertible element in ${M}$;
\end{enumerate}
\end{thm}

\begin{proof} The implication $(a)\Rightarrow (b)$ follows from Lemmas \ref{l Cebysev implies BP-q invert}.\smallskip

$(b)\Rightarrow (a)$ Suppose $x$ is BP-quasi invertible in $M.$ We note that the support tripotent, $r(x),$ of $x$ is complete in $M$, and hence a complete tripotent in $M^{\ast \ast }$ (cf. Lemma \ref{l complete tripotents do not admit orthogonal complements} and comments before it).\smallskip

Suppose that $\mathbb{C}x$ is not a \v{C}eby\v{s}\"{e}v subspace of $M$. By Lemma \ref{l tech Roeberton Singer} there exists an extreme point $\phi$ of the closed unit ball of $M^{\ast }$, $\lambda \in \mathbb{C}\backslash \{0\}$, and $y\in M$ such that $\phi (x)=0$ and $\phi (y)=\Vert y\Vert =\Vert y-\lambda x\Vert$.\smallskip

The support tripotent $\upsilon =s(\phi )$ of $\phi $ in $M^{\ast \ast }$ is a (non-zero) minimal tripotent in $M^{**}$ satisfying $\phi =P_{2}(\upsilon )^* \phi =\phi  P_{2}(\upsilon )$ and $\phi (z) \upsilon
=P_{2}(\upsilon ) (z)$, $\forall z\in M^{\ast \ast }$ (cf. \cite[Proposition 4]{FridRusso}). Therefore, $P_{2}(\upsilon )(x)=\phi (x)\upsilon =0$.\smallskip

We may suppose that $\Vert y\Vert =1$. Since $P_{2}(\upsilon )(y)=\phi (y) \upsilon =\upsilon $, Lemma 1.6 in \cite{FridRusso} implies that $P_{1}(\upsilon )(y)=0$, which shows that $y=\upsilon +P_{0}(\upsilon )y$. We similarly get $P_{1}(\upsilon )(y-\lambda x)=0$ (we simply observe that $\phi(y-\lambda x) = \|y\|= \|y-\lambda x\|=1$). Therefore, $P_{1}(\upsilon )(x)=0$, and  $x=P_{0}(\upsilon )x\in (M^{\ast \ast })_{0}(\upsilon
)=((M^{\ast \ast })_{2}(\upsilon ))^{\perp },$ implying that $x\perp \upsilon$. The equivalent statements in \eqref{ref orthogo} prove that $r(x)\perp \upsilon $, which contradicts Lemma \ref{l complete tripotents do not admit orthogonal complements}.\smallskip
\end{proof}

The above Theorem \ref{t one dimensional Cebysev subspaces in a JBW*-triple} generalizes the previously commented results obtained by Robertson \cite{Robertson1} (compare Theorem \ref{t Robertson 1dimensional Cebysev subspaces in von Neumann algebras}). In order to find a triple version of the reformulation established by Pedersen in \cite[Theorem 2]{Pederesen}, stated as statement $(c)$ in page \pageref{equ reformulation Ronbertson Pedersen}, we recall some notation.\smallskip

For each functional $\varphi$ in the predual of a JBW$^*$-triple $W$, and for each $z$ in $W$ with $\varphi (z) = \|\varphi\|$, and $\|z\|=1$, the mapping $x\mapsto\|x\|_{\varphi}:=(\varphi\{x,x,z\})^{1/2}$ defines a pre-Hilbertian semi-norm on $W$. Moreover, $\varphi\{x,x,w\}=\varphi\{x,x,z\}$ whenever $w\in W$ with $\varphi(w)=\|\varphi\|$ and $\|w\|=1$ (cf. \cite[Proposition 1.2]{BarFri1987}). It is known that \begin{equation}\label{ineq seminorm and module} |\varphi (x)| \leq \|x\|_{\varphi},
 \end{equation} for every $x\in W$ (see \cite[page 258]{BarFri90}). \smallskip

The inequality in \eqref{ineq seminorm and module} together with Lemma \ref{l tech Roeberton Singer} imply the following property: Let $x$ be a non-zero element in a JBW$^*$-triple $M$ such that $\mathbb{C}x$ is a \v{C}eby\v{s}\"{e}v subspace of $M$. Then for each extreme point $\varphi$ of the closed unit ball of $M^*$ we have $\|x \|_{\varphi}\gneqq 0$. It would be interesting to know under what additional hypothesis, the condition $\|x \|_{\varphi}\gneqq 0,$ for every extreme point $\varphi$ of the closed unit ball of $M^*$, implies that $x$ is BP-quasi invertible.

\section{\v{C}eby\v{s}\"{e}v subtriples of JBW$^*$-triples}\label{sec: Cebysev JBW*-subtriples}

In this section, we shall determine the JBW$^*$-subtriples of a JBW$^*$-triple $M$ which are \v{C}eby\v{s}\"{e}v subspaces in $M$. Let us recall that in the case of an infinite dimensional von Neumann algebra $M$, if a finite dimensional von Neumann subalgebra $N$ of $M$ is a \v{C}eby\v{s}\"{e}v subspace in $M$ then $N$ must be one dimensional (compare Theorem \ref{t Robertson Cebysev finite dimensional von Neumann subalgebras in von Neumann algebras} or \cite[Theorem 6]{Robertson1}). Furthermore, an infinite dimensional C$^*$-algebra $A$ admits a finite dimensional $^*$-subalgebra $B$ which is also a \v{C}eby\v{s}\"{e}v in $A$ if and only if $A$ is unital and $B=\mathbb{C} 1$ (cf. \cite[Corollary 1.4]{Robertson2}).  The scarcity of non-trivial \v{C}eby\v{s}\"{e}v C$^*$-subalgebras in general C$^*$-algebras can be better understood with the following result due to G.K. Pedersen: If $A$ is a C$^*$-algebra without unit and $B$ is a \v{C}eby\v{s}\"{e}v C$^*$-subalgebra of $A$, then $A=B$ (compare \cite[Theorem 4] {Pederesen}). \smallskip

The first main difference in the setting of JB$^*$-triples is the existence of \v{C}eby\v{s}\"{e}v JB$^*$-subtriples with arbitrary dimensions. For example, let $E=H$ be a complex Hilbert space regarded as a type 1 Cartan factor with the Hilbert norm and the product \begin{equation}\label{eq triple product Hilbert type 1} \{x,y,z\} = \frac12 (\langle x,y\rangle z+ \langle z,y\rangle x ),
 \end{equation} where $\langle .,.\rangle$ denotes the inner product of $H$. It is known that elements in the unit sphere of a complex Hilbert $H$ space regarded as a type 1 Cartan factor are precisely the complete tripotents of $H$. The \emph{Orthogonal Projection theorem} tells that any closed subspace of $H$ is a \v{C}eby\v{s}\"{e}v  subspace of $H$ and clearly a JB$^*$-subtriple.\smallskip

The following remark provides an additional example.

\begin{remark}\label{remark spin}{\rm Let $E$ be a spin factor with triple product and norm given by $$\J x y z
= \langle x / y \rangle z + \langle z / y \rangle x - \langle x /
\bar z \rangle \bar y,$$ and $\| x\|^2=\langle x / x
\rangle+\sqrt {\langle x / x \rangle^2-|\langle x / \overline x
\rangle|^2}$, respectively, where $x \mapsto \overline{x}$ is a conjugation on $E$, and $\langle ./ .\rangle$ denotes the inner product of $E$. Let $K$ be a closed subspace of $E$ with $\overline{K} = K$. Clearly, $K$ is a JB$^*$-subtriple of $E$. Since $K$ is a closed subspace of the complex Hilbert space $E$, there exists an orthogonal projection $P$ of $E$ onto $K$. Since $E = K\bigoplus H$, where $H= (I-P) (E)$ with $\langle K / H \rangle =0$. Since $\overline{K} = K$, we also have $\overline{H} = H$. Given $\eta\in K$ and $\xi\in H$, it is easy to check that $$\| \eta+\xi\|^2 = \langle \eta+\xi / \eta+\xi \rangle + \sqrt{\langle \eta+\xi / \eta+\xi \rangle^2-|\langle \eta+\xi / \overline{\eta}+\overline{\xi}
\rangle|^2 }$$ $$ = \langle \eta / \eta \rangle + \langle \xi / \xi \rangle + \sqrt{\langle \eta / \eta \rangle^2-|\langle \eta / \overline{\eta} \rangle|^2  + \langle \xi / \xi \rangle^2-|\langle \xi / \overline{\xi}
\rangle|^2} $$ $$\geq \langle \eta / \eta \rangle +\sqrt{ \langle \eta / \eta \rangle^2-|\langle \eta / \overline{\eta} \rangle|^2} = \|\eta\|^2.$$ Moreover, $\| \eta+\xi\| = \| \eta\| $ if and only if $\xi=0$. This shows that $P: E\to E$ is a bi-contractive for the norm $\|.\|$, and for each $x\in E$, $P(x)$ is the unique best approximation of $x$ in $K$. Therefore, $K$ is a \v{C}eby\v{s}\"{e}v  JB$^*$-subtriple of $E$. We observe that the dimensions of $E$ and $K$ can be arbitrarily big.
}\end{remark}

We can present now our conclusions on \v{C}eby\v{s}\"{e}v  JB$^*$-subtriples.\smallskip

The next property of \v{C}eby\v{s}\"{e}v subspaces is probably part of the folklore in the theory of best approximation in normed spaces, but we couldn't find an exact reference.

\begin{lemma}\label{l contractive projection} Let $V$ be a \v{C}eby\v{s}\"{e}v subspace of a normed space $X$. For each $x\in X$, we denote by $c_{_{V}} (x)$ the unique element in $V$ satisfying $\|x-c_{_V}(x)\| = \hbox{dist} (x,V)$. Let $P: X\to X$ be a contractive projection such that $P(V) \subseteq V$. Then $$P\Big( c_{_V}(P(x)) \Big) =  c_{_V}(P(x)),$$ for every $x\in X$. Furthermore, $P(V)$ is a \v{C}eby\v{s}\"{e}v subspace of the normed space $P(X),$ and for each $x\in X,$ $c_{_{P(V)}} (P(x)) = P(c_{_V} (x)).$
\end{lemma}

\begin{proof} Let $x$ be an element in $X$. The condition $\|P\|\leq 1$ implies that $$\Big\|P(x) - P\Big( c_{_V}(P(x)) \Big) \Big\| \leq \Big\|P(x) - c_{_V}\Big(P(x)\Big) \Big\| =  \hbox{dist} (P(x),V).$$ The element $P\Big( c_{_V}(P(x)) \Big)\in P(V) \subseteq V$. Thus, the uniqueness of the best approximation in $V$ proves that $P\Big( c_{_V}(P(x)) \Big) =  c_{_V}(P(x))$. The rest is clear.
\end{proof}

\begin{prop}\label{prop A} Let $F$ be a \v{C}eby\v{s}\"{e}v  JB$^*$-subtriple of a JB$^*$-triple $E$. Suppose $e$ is a non-zero tripotent in $F$. Then $E_0 (e) = F_0 (e)$. Consequently, every complete tripotent in $F$ is complete in $E$.
\end{prop}

\begin{proof} Since $e$ is a tripotent in $F$ and the latter is a JB$^*$-subtriple of $E$, $e$ is a tripotent in $E$ and $F_0 (e) \subseteq E_0(e)$. Arguing by contradiction, let us assume that there exists $b\in E_0 (e) \backslash F_0 (e) =  E_0 (e) \backslash F \neq \emptyset$. Since $ \hbox{dist} (b,F) >0$ and $F$ is a \v{C}eby\v{s}\"{e}v subspace, there exists a unique $c_{_F} (b) \in F$ such that $\|b-c_{_F}(b)\| = \hbox{dist} (b,F)$.\smallskip

Since $P_0 (e) (F) \subseteq F$ and $P_0 (e) (b) =b$, Lemma \ref{l contractive projection} implies that $$P_0 (e) (c_{_F}(b)) = c_{_F}(b)\in F_0 (e).$$  Having in mind that $e\in E_2 (e)\perp  E_0 (e)\ni b- c_{_F}(b)$, we deduce, via \eqref{eq orthogonal are M-orthogonal}, that $$\| b-  c_{_F}(b) -\lambda e \| = \max\{ \| b-  c_{_F}(b) \|, |\lambda|  \} = \| b-  c_{_F}(b) \| = \hbox{dist} (b,F),$$ for every $|\lambda|\leq \hbox{dist} (b,F)$. This contradicts the uniqueness of the best approximation, $ c_{_F}(b),$ of $b$ in $F$, because $c_{_F}(b) +\lambda e \in F$ for every $|\lambda|\leq \hbox{dist} (b,F)$.
\end{proof}

\begin{prop}\label{prop B} Let $F$ be a \v{C}eby\v{s}\"{e}v  JB$^*$-subtriple of a JB$^*$-triple $E$. Suppose $e$ is a tripotent in $F$ with $F_0(e) = \{e\}^{\perp}_{F}\neq 0$. Then $E_2 (e) = F_2 (e)$.
\end{prop}

\begin{proof} Clearly $F_2 (e) \subseteq  E_2 (e)$. We have to show that $E_2 (e) \subseteq  F_2 (e).$ Suppose, on the contrary, that $E_2 (e) \backslash F_2 (e) = E_2 (e) \backslash F\neq \emptyset$. Pick $b\in E_2 (e) \backslash F$. Since $F$ is a \v{C}eby\v{s}\"{e}v subspace of $E$, there exists a unique $c_{_{F}} (b)\in F$ satisfying $\| b-  c_{_F}(b) \| = \hbox{dist} (b,F)>0.$\smallskip

By Lemma \ref{l contractive projection} applied to $P= P_2 (e)$, $X= E$ and $V=F$, we deduce that $P_2 (e) ( c_{_F}(b)) =  c_{_F}(b).$\smallskip

By hypothesis, $F_0(e) = \{e\}^{\perp}_{F}\neq 0$. So, there exists a norm-one element $z\in F_0 (e)$. The conditions $b,\in E_2 (e)$, $c_{_F} (b)\in F_2 (e)$ and $z\in F_0 (e)$ combined with \ref{eq orthogonal are M-orthogonal} give $$\| b-  c_{_F}(b) -\lambda z \| = \max\{ \| b-  c_{_F}(b) \|, |\lambda|  \} = \| b-  c_{_F}(b) \| = \hbox{dist} (b,F),$$ for every $|\lambda|\leq \hbox{dist} (b,F)$, which contradicts the uniqueness of the best approximation of $b$ in $F$ because $c_{_F}(b) -\lambda z\in F$, for every $\lambda$ in the above conditions.
\end{proof}

Let $e$ and $v$ be tripotents in a JB$^*$-triple $E$. We shall say that $v\leq e$, when $e-v$ is a tripotent in $E$ with $e-v\perp v$ (compare the notation in \cite{FridRusso}).\smallskip

Let $E$ be a JB$^*$-triple. A subset $S \subseteq E$ is said to be \emph{orthogonal} if $0 \notin S$ and $x \perp y$ for every $x
\neq y$ in $S$. The minimal cardinal number $r$ satisfying $card(S) \leq r$ for every orthogonal subset $S \subseteq E$ is called the \emph{rank} of $E$ (and will be denoted by $r(E)$). Given a tripotent $e\in E$,  the rank of the Peirce-2 subspace $E_2(e)$ will be called the rank of $e$.\smallskip

Theorem 3.1 in \cite{BeLoPeRo} combined with Proposition 4.5.(iii) in \cite{BuChu} assure that a JB$^*$-triple is reflexive if and only if it is isomorphic
to a Hilbert space if, and only if, it has finite rank.\smallskip

Suppose $E$ is a rank-one JB$^*$-triple. The above comments show that $E$ is reflexive and hence a JBW$^*$-triple. Let $e$ be a complete tripotent in $E$. Since the rank of $e$ is smaller than the rank of $E$, we deduce that $e$ is a minimal tripotent in $E$. Proposition 3.7 in \cite{Burgos2} and its proof show that $E= \{e\}^{\perp\perp} = \{0\}^{\perp}$ is a rank-one Cartan factor of the form $L(H,\mathbb{C})$, where $H$ is a complex Hilbert space or a type 2 Cartan factor $II_{3}$ (it is known that $II_{3}$ is JB$^*$-triple isomorphic to a 3-dimensional complex Hilbert space). We have proved the following:

\begin{lemma}\label{l rank one JB*-triples are type 1 Cartan factors Hilbert} Every JB$^*$-triple of rank one is JB$^*$-isomorphic {\rm(}and hence isometric{\rm)} to a complex Hilbert space regarded as a type 1 Cartan factor.$\hfill\Box$
\end{lemma}

The above result is also stated in \cite[Corollary in page 308]{DaFri}.\smallskip

We have already commented that orthogonal elements are $M$-orthogonal in the sense of the geometric theory of Banach spaces (see \eqref{eq orthogonal are M-orthogonal}). We shall state next another results of geometric nature. Let $u$ and $v$ be two non-zero tripotents in a JB$^*$-triple $E$. We recall that $u$ and $v$ are \emph{colinear} (written $u\top v$) when $u\in E_1 (v)$ and $v\in E_1 (u)$ (cf. \cite[page 296]{DaFri}). Suppose $u\top v$ in $E$. Clearly, the JB$^*$-subtriple $E_{u,v}$ of $E$ generated by $u$ and $v$ is algebraically isomorphic to $\mathbb{C} u\otimes \mathbb{C} v$. We observe that $u$ and $v$ are minimal colinear tripotents in $E_{u,v}$. It follows from \cite[Proposition 5]{FridRusso} that $E_{u,v}$ is JB$^*$-triple isomorphic and hence isometric to $M_{1,2} (\mathbb{C})$ (regarded as a type 1 Cartan factor). We consequently have \begin{equation}\label{eq colinear are Hilbert 2-summand} \|\lambda u + \mu v\| = \left(|\lambda|^2 +|\mu|^2\right)^{\frac12},
\end{equation} for every $\lambda, \mu \in \mathbb{C}$. It should be also noted here that, in a Hilbert space $F$ regarded as a type 1 Cartan factor with product given in \eqref{eq triple product Hilbert type 1}. In this case, the tripotents of $F$ are precisely the elements in its unit sphere, and the relation of being Hilbert-orthogonal is exactly the relation of colinearity in terms of the triple product.\smallskip

We have shown several examples of Hilbert spaces (regarded as a type 1 Cartan factor) which are \v{C}eby\v{s}\"{e}v JB$^*$-subtriples of JB$^*$-triples of rank one and two. We present next more examples of Hilbert spaces which are \v{C}eby\v{s}\"{e}v JB$^*$-subtriples of JB$^*$-triples having a bigger rank. The first example is a construction with classical Banach spaces and the second one is an isometric translation to the setting of JB$^*$-triples.

\begin{remark}\label{example Hilbert Cebysev in rectangular Banach spaces} Let $H$ be complex Hilbert space of dimension $2$ with norm denoted by $\|.\|_2$. We consider the Banach space $\displaystyle X = \overbrace{H \oplus^{\ell_\infty} \ldots \oplus^{\ell_\infty} H}^{(n)}$ ($n\geq 2$). Let $\{\xi_1,\xi_2\}$ be an orthonormal basis of $H$. Each $h\in H$ writes uniquely in the form $h =  \lambda_1 \xi_1 +  \lambda_2 \xi_2$. Let $V$ denote the 2-dimensional subspace of $X$ generated by the vectors $e_1=(\xi_1,\ldots, \xi_1)$ and $e_2=(\xi_2,\ldots, \xi_2)$. That is, every vector in $V$ writes in the form $\lambda e_1 + \mu e_2$ Clearly, $$\| \lambda e_1 + \mu e_2 \| = \|\lambda (\xi_1,\ldots, \xi_1) + \mu (\xi_2,\ldots, \xi_2) \|_2 $$ $$= \max_{i=1,\ldots,n} \|\lambda \xi_1 + \mu \xi_2 \|_2 = \sqrt{|\lambda|^2 +|\mu|^2} ,$$ and hence $V$ is isometrically isomorphic to a Hilbert space.\smallskip

We claim that $V$ is a \v{C}eby\v{s}\"{e}v subspace of $X$. Indeed, let $x=(h_1,\ldots, h_n)$ be an element in $X$ and let $\lambda e_1 + \mu e_2\in V$. We write $h_i = \lambda^{i}_1 \xi_1 +  \lambda^{i}_2 \xi_2$. We write the formula for the distance from $x$ to $V$ in the form: $$\hbox{dist} (x, V)^2= \inf_{\lambda,\mu \in \mathbb{C}} \| (h_1,\ldots, h_n) - \lambda e_1 - \mu e_2 \|^2 $$ $$= \inf_{\lambda,\mu \in \mathbb{C}} \max_{i=1,\ldots,n} \|\lambda^{i}_1 \xi_1 +  \lambda^{i}_2 \xi_2 - \lambda \xi_1 -\mu \xi_2 \|_2^2$$
$$ = \inf_{\lambda,\mu \in \mathbb{C}} \max_{i=1,\ldots,n} \left( |\lambda_1^{i}- \lambda |^2 + |\lambda^{i}_2 -\mu |^2 \right)^{\frac12}= \inf_{\lambda,\mu \in \mathbb{C}} \max_{i=1,\ldots,n} \hbox{dist}_{_{\mathbb{C}^2}} ((\lambda_1^{i},\lambda_2^{i}),(\lambda,\mu)).$$

Our problem is equivalent to determine a point $(\lambda,\mu)\in \mathbb{C}^{2}$ so that the maximum Euclidean distance from $(\lambda,\mu)$ to the points $(\lambda_1^{i}, \lambda^{i}_2) \in \mathbb{C}^{2}$ ($i=1,\ldots,n$) is minimized, where $\mathbb{C}^{2}$ is equipped with the Euclidean distance $\|(\lambda , \mu ) \|_2 = \sqrt{|\lambda|^2 +|\mu|^2}$. This problem is commonly called ``\emph{the Euclidean delivery problem}'' or ``\emph{the min-max location problem}'' or ``\emph{the minimum covering sphere problem}''. It is known that an equivalent reformulation of the problem is:
$$\hbox{Min} \{\rho : (\lambda,\mu)\in \mathbb{C}^2, \ \rho>0, \ \|(\lambda_1^{i},\lambda_2^{i})-(\lambda,\mu)\|_2\leq \rho, \ \forall i\}.$$
The goal is to find the circle of center $(\lambda,\mu)\in \mathbb{C}^{2}$ of smallest radius $\rho$ that encloses all the points $(\lambda_1^{i}, \lambda^{i}_2) \in \mathbb{C}^{2}$ ($i=1,\ldots,n$).\smallskip

It is well known that a solution to the the minimum covering sphere problem always exists, the center $(\lambda,\mu)$ and the radius $\rho$ are unique (cf. \cite{HearnVijay82}, \cite{Francis67}). This shows that every element $x=(\lambda^{1}_1 \xi_1 +  \lambda^{1}_2 \xi_2,\ldots, \lambda^{n}_1 \xi_1 +  \lambda^{n}_2 \xi_2)$ in $X$ admits a unique best approximation in $V$, which proves the claim.
\end{remark}

\begin{remark}\label{example Hilbert Cebysev in rectangular} Let $e$ and $u$ be two colinear complete tripotents in a JB$^*$-triple $E$. Let us assume that we can find two sets $\{e_1,\ldots,e_n\}$ and $\{u_1,\ldots,u_n\}$ of mutually orthogonal tripotents in $E_2(e)$ and $E_2(u)$, respectively, such that $e_i\top u_i$, for all $i$, and $u_i\perp e_j$, for every $i\neq j$. Take, for example, $E=M_{n\times (2n)} (\mathbb{C})$, $e=\sum_{i=1}^n w_{i,i}$, $u=\sum_{i=1}^n w_{i,i+n}$, $e_i= w_{i,i}$ and $u_i = e=w_{i,i+n}$, where $w_{i,j}$ is the matrix with entry 1 at the position $i,j$ and zero elsewhere.\smallskip

Let $F$ be the JB$^*$-subtriple of $E$ generated by $\{e_1,\ldots, e_n, u_1,\ldots, u_n\}$, and let $W$ be the closed JB$^*$-subtriple of $F$ generated by $\{e,u\}.$ For each $i\in \{1,\ldots, n\}$, $e_i \top u_i$ and hence $$\|\lambda_i e_i + \mu_i u_i \| = \sqrt{|\lambda_i|^2+|\mu_i|^2},$$ that is, the subtriple, $F_i$, generated by $e_i$ and $u_i$ is a 2-dimensional complex Hilbert space (cf. \eqref{eq colinear are Hilbert 2-summand}). Since, for each $i\neq j$, $\{e_i,u_i\}\perp \{e_j,u_j\}$ ($F_i\perp F_j$), we deduce from \eqref{eq orthogonal are M-orthogonal} that $\|x_i+ x_j \| = \max\{\|x_i\|,\|x_j\|\}$, for every $x_i\in F_i, x_j\in F_j$, $i\neq j$. Having in mind that $F = F_1\oplus^{\ell_{\infty}}\ldots \oplus^{\ell_{\infty}} F_n $, and $F_i \equiv \ell_2^2$, we can easily see that $F$ is isometrically isomorphic to the space $X$ in Remark \ref{example Hilbert Cebysev in rectangular Banach spaces}. It is also easy to see that under the natural isometric identification of $F$ and $X$ in Remark \ref{example Hilbert Cebysev in rectangular Banach spaces}, the JB$^*$-subtriple $W$ is identified with the subspace $V$ in that Remark. Therefore, it follows that $W$ is a \v{C}eby\v{s}\"{e}v JB$^*$-subtriple of $F$. The JB$^*$-triple $F$ has been constructed to have rank $n.$
\end{remark}

The theorem describing the \v{C}eby\v{s}\"{e}v  JBW$^*$-subtriples of a JBW$^*$-triple can be stated now. We shall show that the examples given in Remark \ref{remark spin} and the comments before it are essentially the unique examples of non-trivial \v{C}eby\v{s}\"{e}v JBW$^*$-subtriples.

\begin{thm}\label{t Cebysev  JBW*-subtriples} Let $N$ be a non-zero \v{C}eby\v{s}\"{e}v JBW$^*$-subtriple of a JBW$^*$-triple $M$. Then exactly one of the following statements holds:\begin{enumerate}[$(a)$]\item $N$ is a rank one JBW$^*$-triple with dim$(N)\geq 2$ (i.e. a complex Hilbert space regarded as a type 1 Cartan factor). Moreover, $N$ may be a closed subspace of arbitrary dimension and $M$ may have arbitrary rank;
\item $N= \mathbb{C} e$, where $e$ is a complete tripotent in $M$;
\item $N$ and $M$ have rank two, but $N$ may have arbitrary dimension;
\item $N$ has rank greater or equal than three and $N=M$.
\end{enumerate}
\end{thm}

\begin{proof} We can always find a complete tripotent $e$ in $N$ (see the comments in page \pageref{page complete trip}). Proposition \ref{prop A} implies that $e$ is complete in $M$ (i.e. $M_0 (e) =\{0\}$). We have three possibilities:\begin{enumerate}[$(i)$]\item $e$ has rank one in $N$;
\item $e$ has rank  $2$ in $N$;
\item $e$ has rank greater or equal than $3$ in $N$.
\end{enumerate} \smallskip

$(i)$ Suppose first that $e$ has rank one in $N$. In this case, $e$ is a minimal and complete tripotent in $N$. Therefore, $N$ is a complex Hilbert space regarded as a type 1 Cartan factor (cf. Lemma \ref{l rank one JB*-triples are type 1 Cartan factors Hilbert} or Proposition 3.7 in \cite{Burgos2}).\smallskip

The examples given before Remark \ref{remark spin} and in Remark \ref{example Hilbert Cebysev in rectangular} show that $N$ may have arbitrary dimension and $M$ may have rank as big as desired.\smallskip

$(ii)$ We assume now that $e$ has rank  $2$ in $N$. Then there exist two non-zero minimal, mutually orthogonal tripotents $e_1,e_2\in N$ with $e= e_1+e_2$. Propositions \ref{prop A} and \ref{prop B} show that $M_2 (e_j)= N_2 (e_j)$, and $M_0 (e_j)= N_0 (e_j)\neq \{0\}$, for every $j$ in $\{1,2\}$. Since $M_2 (e_j)= N_2 (e_j) = \mathbb{C} e_j$, we deduce that $e_1$ and $e_2$ are minimal tripotents in $M$. We also know that $e= e_1 +e_2$ is a complete in $M$ (i.e. $M= M_2 (e) \oplus M_1(e)$), which proves that $M$ has rank two. The statement concerning the dimension of $N$ follows from the example in Remark \ref{remark spin}.\smallskip

$(iii)$ Suppose now that $e$ has rank greater or equal than $3$ in $N$. We shall show that $M=N$. Under the present assumptions, we can find three non-zero mutually orthogonal tripotents $e_1, e_2, e_3$ with $e_1+e_2+e_3 =e.$ Clearly, $N_0 (e_j+e_k)\neq \{0\},$ for every $k\neq j$ in $\{1,2,3\}$. Propositions \ref{prop A} and \ref{prop B} assure that $M_2 (e_j+ e_k) = N_2 (e_j+ e_k)$, $M_0 (e_j+ e_k)= N_0 (e_j+ e_k)$, $M_2 (e_j)= N_2 (e_j)$, and $M_0 (e_j)= N_0 (e_j)$, for every $k\neq j$ in $\{1,2,3\}$. In the Peirce decomposition $$M = M_2 (e_1) \oplus M_1 (e_1) \oplus M_0 (e_1),$$ we have $M_2(e_1) = N_2 (e_1)$ and $M_0(e_1) = N_0 (e_1)$. Pick $x\in M_1 (e_1)$. Since $e_1\perp e_j$ ($j=2,3$) we have $M_1 (e_1) \cap M_2 (e_j)= \{0\}$ for every $j=2,3$. Therefore $$x = P_1 (e_2) (x) + P_0(e_2) (x),$$ where $P_0(e_2) (x)\in M_0 (e_2) = N_0 (e_2)\subseteq N$ and $P_1(e_2) (x) \in P_1 (e_2) (N_1 (e_1))$. Since $$\frac12   P_0(e_2) (x) +\frac12  P_1(e_2) (x) = \frac12 x = \{e_1,e_1,  x\} $$ $$= \{e_1,e_1,  P_0(e_2) (x)\} +\{e_1,e_1,   P_1(e_2) (x)\},$$ it follows from Pierce rules that $$\frac12   P_1(e_2) (x) = \{e_1,e_1,  P_1(e_2) (x)\}, $$ and hence $ P_1(e_2) (x) \in M_1 (e_1) \cap M_1 (e_2)$. The condition $e_1\perp e_2$ leads us to $\{e_1+e_2,e_1+e_2, P_1(e_2) (x)\} = P_1(e_2) (x)$, which means that $$P_1(e_2) (x)\in M_2 (e_1+e_2) = N_2 (e_1+e_2)\subseteq N.$$ We have therefore shown that $x = P_1 (e_2) (x) + P_0(e_2) (x)\in N$, which implies that $M_1 (e_1) \subseteq N$ and consequently $M= N$. This concludes the proof.
\end{proof}

Let us recall that a C$^*$-algebra is reflexive if and only if it if finite dimensional (cf.  \cite[Proposition 2]{SakWC}). Consequently, a C$^*$-algebra has finite rank if and only if it is finite dimensional. It is further known that a C$^*$-algebra $A$ has rank one if, and only if, $A= \mathbb{C} 1.$  In particular, the result established by Robertson in \cite[Theorem 6]{Robertson1} (see Theorem \ref{t Robertson Cebysev finite dimensional von Neumann subalgebras in von Neumann algebras}) is a direct consequence of our last theorem.

\begin{corollary}\label{cor JBW*-subtriples of von Neumann algebras} Let $M$ be an infinite dimensional von Neumann algebra. Let $N$ be a \v{C}eby\v{s}\"{e}v von Neumann subalgebra of $M$. Then $N = \mathbb{C} 1$ or $M=N$. $\hfill\Box$
\end{corollary}

We have already seen that, for each natural $n$, we can find a complex Hilbert space (of dimension 2) which is a \v{C}eby\v{s}\"{e}v JB$^*$-subtriple of a JB$^*$-triple having rank $n$. It is natural to ask whether we can find a precise description of those complex Hilbert spaces which are \v{C}eby\v{s}\"{e}v JBW$^*$-subtriples of a JBW$^*$-triple. Another general question that remains open in this paper is the following:

\begin{problem}\label{problem2} Determine the \v{C}eby\v{s}\"{e}v JB$^*$-subtriples of a general JB$^*$-triple.
\end{problem}

\end{document}